\documentclass[12pt,reqno]{amsart}
\usepackage{tikz,dsfont}
\usepackage{amsmath, amssymb, graphicx, mathrsfs, hyperref}
\usepackage{pdfpages}
\usepackage{verbatim} 
\usepackage{enumerate}
\newtheoremstyle{dotless}{}{}{\itshape}{}{\bfseries}{}{ }{} 
\theoremstyle{dotless}

\newtheorem{theorem}{Theorem}[section]

\newtheorem{lemma}{Lemma}

\newtheorem{prop}{Proposition}

\newtheorem{cor}{Corollary}

\begin{document} 

\title{Sieving intervals and Siegel zeros}

\author{ Andrew Granville}

\address{D{\'e}partment  de Math{\'e}matiques et Statistique,   Universit{\'e} de Montr{\'e}al, CP 6128 succ Centre-Ville, Montr{\'e}al, QC  H3C 3J7, Canada.}
   \email{andrew@dms.umontreal.ca}

\begin{abstract}
Assuming that there exist (infinitely many) Siegel zeros, we show that the (Rosser-)Jurkat-Richert bounds in the linear sieve cannot be improved, and similarly look at Iwaniec's lower bound  on Jacobsthal's problem, as well as minor improvements to the Brun-Titchmarsh Theorem. We also deduce an improved (though conditional) lower bound on the longest gaps between primes, and rework Cram\'er's heuristic in this situation to show that we would expect gaps around $x$ that are significantly larger than $(\log x)^2$.
\end{abstract}

\thanks{
Thanks to John Friedlander, Dimitris Koukoulopoulos and James Maynard for helpful responses to emailed questions, and Kevin Ford for some useful comments, especially concerning the consequences of Siegel zeros for Cram\'er's conjecture discussed after the statement of Corollary 2, as well as on the exposition.}

\date{}

\maketitle

\section{The history of the problem}

We are interested in determining sharp upper and lower bounds for the number of integers  with no small prime factors in a short interval; specifically, estimates on
\[
S(x,y,z) := \#\{ n\in (x,x+y]:\ (n,P(z))=1\}
\]
where $P(z):=\prod_{p\leq z} p$. This question is an example of problems that can be attacked by the \emph{small sieve} (as is, for example, estimating pairs of integers that differ by 2, which have no small prime factors).   This more general set up goes as follows:

We begin with a set of integers $\mathcal A$ (of size $X$) to be \emph{sieved} (in our case the integers in the interval $(x,x+y]$).
It is important that the proportion of elements of $\mathcal A$ that is divisible by integer $d$ is very close to a multiplicative function (in $d$):
If $\mathcal A_d:=\{ a\in \mathcal A:\ d|a\}$ then we write
\[
\# \mathcal A_d = \frac{g(d)}d X + r(\mathcal A,d)
\]
where $g(d)$ is a multiplicative function, which is more-or-less bounded by some constant $\kappa>0$ on average over primes $p$, even in short intervals (in our case each $g(p)=1$):
\[
\prod_{y<p\leq z} \bigg( 1 -\frac{g(p)}p \bigg)^{-1} \leq \bigg(  \frac{\log z}{\log y} \bigg)^\kappa \bigg( 1 +O\bigg(\frac{1}{\log y} \bigg)\bigg)
\]
(and, in our case, Merten's Theorem allows us to take $\kappa=1$, the \emph{linear sieve}); and 
$r(\mathcal A,d)$ is an error term that must be small on average (in our case each   $|r(\mathcal A,d)|\leq 1$):
\[
\sum_{\substack{ d|P(z)\\ d\leq D}} |r(\mathcal A,d)| \ll_A \frac X{(\log X)^A}
\]
for any $A>0$ where $D=X^\theta$ for some $\theta>0$ (in our case we can take any $\theta<1$). The goal in sieve theory is to estimate
\[
S(\mathcal A,z) := \{ n\in \mathcal A:\ (n,P(z))=1\},
\]
which ``on average'' equals 
\[
G(z)X   \text{ where } G(z):= \prod_{p\leq z} \bigg( 1 -\frac{g(p)}p \bigg),
\]
though here we are   interested in the extreme cases; that is, the smallest and largest values of 
$S(\mathcal A,z)$ under these hypotheses.

In 1965, Jurkat and Richert \cite{JR} showed for $\kappa=1$ that if $X=z^u$ then 
\begin{equation} \label{eq: JR theorem}
(f(u)+o(1))  \cdot G(z)X \leq S(\mathcal A,z)  \leq (F(u)+o(1)) \cdot G(z)X    ,
\end{equation}
where $f(u)= e^\gamma(\omega(u) -\frac{\rho(u)}u)$ and $F(u)= e^\gamma(\omega(u) +\frac{\rho(u)}u)$, and
$\rho(u)$ and $\omega(u)$ are the Dickman-de Bruijn and Buchstab functions, respectively.\footnote{According to Selberg \cite{Se1}, section 5, Rosser had proved this result ten years earlier in unpublished notes, with a proof that looks superficially different, but Selberg felt was probably fundamentally the same.} One can define these functions directly by
\[
f(u)=0 \text{ and }  F(u)=\frac{2e^\gamma}u \text{ for } 0<u\leq 2
\]
(in fact $F(u)=\frac{2e^\gamma}u$ also for $ 2<u\leq 3$) and
\[
 f(u)= \frac 1u \int_1^{u-1} F(t) dt   \text{ and }  F(u)=\frac {2e^\gamma} u  + \frac 1u \int_2^{u-1} f(t) dt   \text{ for all } u\geq 2.
\]

Iwaniec  \cite{Iw1} and Selberg \cite{Se1} showed that this result is ``best possible'' by noting that the sets
\[
\mathcal A^\pm =\{ n\leq x: \lambda(n)=\mp 1\}  
\]
 where $\lambda(n)$ is Liouville's function (so that $\lambda(\prod_p p^{e_p})= (-1)^{\sum_p e_p}$) satisfy the above hypotheses, with 
 \begin{equation} \label{eq: Selbergexample}
 S(\mathcal A^-,z) =(f(u)+o(1)) \cdot G(z) \# \mathcal A^- \text{ and } S(\mathcal A^+,z) =(F(u)+o(1))\cdot  G(z) \# \mathcal A^+.
 \end{equation}
  If $u\leq 2$ then $S(\mathcal A^-,z)$  counts the  integers $\leq x$, which have an even number of prime factors, all $>z\geq x^{1/2}$: The only such integer is $1$, and therefore $S(\mathcal A^-,z)=1$. Thus sieving with these hypotheses (and minor variants) one  cannot   detect primes (this is the so-called  \emph{parity phenomenon}).
 
 In order to detect primes in $[1,x]$ or just to better understand $S(x,y,z)$, we need some other techniques. The purpose of this article is to show that, even so, it will be a (provably) difficult  task. We know that if $y=z^u$ then
 \[
 (f(u)+o(1)) \cdot G(z)y \leq S(x,y,z) \lesssim F(u) \cdot G(z)y , \text{ where } G(z)=\prod_{p\leq z} \bigg( 1 -\frac{1}p \bigg),
 \]
 since sieving an interval is, as we discussed, an example of this more general linear sieve problem.  In \cite{Se2}, Selberg asked ``\textsl{is it possible that these quantities [the best possible upper and lower bounds for sieving an interval] behave significantly differently [from the bounds in the general linear sieve problem] ? ... We do not know the answer}''.
 
 We will show that if there are infinitely many Siegel zeros then the general bounds are also best possible for the problem of sieving intervals.

 \begin{cor} \label{cor: SiegelGeneralv} Assume that there are infinitely many Siegel zeros.  For each fixed $v>1$,  there exist
 arbitrarily large $x,X,y,z$ with $y=z^v$ such that 
 \[
 S(x,y,z) = (F(v)+o(1))G(z)y \text{ and } S(X,y,z) = (f(v)+o(1))G(z)y 
  \]
 \end{cor}
 
 We will explain   what exactly we mean by ``Siegel zeros'' in the next section.
 
 Siebert in \cite{Si}   proved a similar result though with a slightly broader sieve problem (he allowed sieving  arithmetic progressions), and he obtained a slightly weaker conclusion because he did not realize that an estimate as strong as \eqref{eq: PNTAPs} was at his disposal. Our proof is also a little easier since we determine our estimates in Corollary \ref{cor: SiegelGeneralv} by calculating in terms of  the Selberg-Iwaniec $\mathcal A^\pm $ examples, rather than calculating complicated explicit expressions for $F(v)$ and $f(v)$.  One can deduce from Corollary \ref{cor: SiegelGeneralv} that if $\pi(qx;q,a)\leq (2-\epsilon) \frac q{\phi(q)} \frac{x}{\log x}$ for $x=q^A$ with $q$ sufficiently large then there are no Siegel zeros, which reproves a result of Motohashi \cite{Mo}.\footnote{Actually Motohashi showed the more precise bound that if $\pi(qx;q,a)\leq (2-\epsilon) \frac q{\phi(q)} \frac{x}{\log x}$ and $L(\beta, \chi_q)=0$ with $\beta\in \mathbb R$ then $\beta<1-\frac{c\epsilon}{\log q}$. We   obtain a similar result in the first part of Proposition \ref{prop;Motohashi}. Motohashi remarked that ``an extension of [his] theorem in a  direction similar to Siebert is quite possible''..}

 Corollary \ref{cor: SiegelGeneralv}   implies that if $1<v\leq 2$ then there are arbitrarily large $X,y,z$ with $y=z^v$ for which
 \[
 S(X,y,z) = o\bigg( \frac y{\log y}\bigg) .
 \]
since $f(2)=0$. However $f(u)>0$ for $u>2$ and so it is of interest to understand $S(x,y,z)$ when $y=z^{2+o(1)}$. The key result in this range is 
 due to Iwaniec \cite{Iw2}\footnote{A slightly weaker version of this result more-or-less follows from Iwaniec's much earlier Theorem 2 in \cite{Iw1}.} who showed that if $y\gg z^2$ then
\[
S(x,y,z)\geq  \frac {4y}{(\log y)^2} \cdot (\log (y/z^2)-O(1)) .
\]
(In fact uniformly for $2<u\leq 3$ he proved that  
 \[
S(x,y,z)\geq   \bigg( f(u) - \frac c{\log y}\bigg) \cdot   \prod_{p\leq z} \bigg( 1 -\frac{1}p \bigg) y,
\]
where $f(u)=\frac{2e^\gamma\log ( u-1)}u$ in this range.)

 \begin{prop} \label{prop; Iwaniec}
 Suppose that there is an infinite sequence of primitive real characters $\chi$  mod $q$ such that there is an exceptional zero $\beta=\beta_q$ of each $L(s,\chi)$. For each, there exists a corresponding value of $y$ such that if  $y^{1-\epsilon}>z>y^{1/2-o(1)}$ then there exists an integer $X$ for which
 \[
 S(X,y,z)\lesssim  \frac {4y}{(\log y)^2}  \log^+ (qy/z^2 )+ (1-\beta_q) y 
 \]
 where $\log^+t=\max\{ 0,\log t\}$. We can take $y=q^{A-1}$ with $A\to \infty$ as slowly as we like.
 \end{prop}

 We will also deduce the following:
 
 \begin{cor} \label{cor: PrimeGaps}
 Suppose that there are infinitely Siegel zeros $\beta$ with $1-\beta<\frac 1{(\log q)^B}$ for some integer $B\geq 1$.
Then there are infinitely primes $p_n$ (where $p_n$ is $n$th smallest prime) for which
 \[
 p_{n+1}-p_n\gg   \log p_n (\log\log p_n)^{B-1}.
 \]
 \end{cor}
  
One can obtain longer gaps between primes from the proof if there are Siegel zeros even closer to $1$. Pretty much the same lower bounds on the maximal gaps between primes were recently given by Ford \cite{For} (by fundamentally the same proof).
Unconditionally proved lower bounds on the largest prime gaps \cite{FGM}  are slightly smaller than $\log p_n (\log\log p_n)$, and the techniques used seem unlikely to be able to prove much more than $\log p_n (\log\log p_n)^2$, so we are in new territory here once $B>3$.

It is believed that  there are gaps $p_{n+1}-p_n\gg   (\log p_n)^2$ but that seems out of reach here. However, after the proof of Corollary \ref{cor: PrimeGaps}, we show that the standard heuristic (see e.g. \cite{GL}) implies that if there are infinitely many Siegel zeros then
\[
\limsup_{x\to \infty}  \frac{ \max_{p_n\leq x} p_{n+1}-p_n }{(\log x)^2} \to \infty
 \]
 (as suggested to me by Ford).
Moreover under the hypothesis of  Corollary \ref{cor: PrimeGaps} with $B=2/\epsilon-1$, this same heuristic implies that
there are infinitely $p_n$ for which
\[
p_{n+1}-p_n  \gg (\log p_n)^2 (\log\log p_n)^{1-\epsilon}.
\]

 \emph{Jacobsthal's function} $J(m)$ is defined to be the smallest integer $J$ such that every $J$ consecutive integers contains one which is coprime to $m$. 
 Therefore if $m=P(z)$ then $J(m)$ is the smallest integer $y$ for which $S(x,y,z)\geq 1$ for all $x$.  It is not difficult to show that $J(m)\sim \frac{m}{\phi(m)} \omega(m)$ for almost all integers $m$, but we are most interested in $\max_{m\leq M} J(m)$, and believe that the maximum occurs either for the largest $m=P(z)\leq M$ or for another integer that has almost as many prime factors.
 
 Iwaniec's result above establishes that 
$ J(P(z))\ll z^2$; by the prime number theorem which means that $J(m)\ll (\omega(m)\log \omega(m))^{2}$ (where 
 $\omega(m)$ denotes the number of distinct prime factors of $m$), and Iwaniec \cite{Iw2} deduced (cleverly) that this upper bound then holds for all integers $m$.
  The proof in  \cite{FGM}  implies that if $m=P(z)$ then $J(m)\gg \omega(m) (\log \omega(m))^{2}\frac{\log_3 \omega(m)}{\log_2 \omega(m)}$, and the methods there suggest the conjecture that the largest $J(m)$ gets is something like  $ \omega(m) (\log \omega(m))^{3+o(1)}$.  However our proof of Corollary \ref{cor: PrimeGaps}
implies that if there are infinitely Siegel zeros $\beta$ with $1-\beta<\frac 1{(\log q)^B}$ for some integer $B\geq 1$, then there exist integers $m$ with
$J(m)\gg  \omega(m) (\log \omega(m))^{B}$; and therefore this conjecture is untrue if $B$ can be taken to be $>3$.
(This is also easily deduced from the discussion in Ford \cite{For}.)

 Corollary \ref{cor: SiegelGeneralv}  implies that if $1\leq v\leq 3$ (that is, $z\geq y^{1/3+o(1)}$) then\footnote{Selberg explains in \cite{Se3}, section 18 that, before his elementary proof of the prime number theorem, when analyzing what prevented him from substantially improving the upper bound $ < \frac {2y}{\log y}$, he found all of the main contribution to the terms in the sieve sum came from integers with an odd number of (large) prime factors, and therefore came up with the $\mathcal A^\pm $ examples. Bombieri   greatly expanded on this phenomenon in \cite{Bo}.}
 \[
S(x,y,z) \sim     \frac {2y}{\log y}.
 \]

A subset $A$ of the integers in $[0,y]$ has \emph{length} $\leq y$, and is \emph{admissible} if for every prime $p$ there is a residue class mod $p$ that does not contain an element of $A$. It is believed that the largest admissible set of length $y$ contains $\sim \frac {y}{\log y}$ elements. It is worth emphasizing that our results show that this belief is untrue if there are Siegel zeros:

  \begin{cor} \label{cor: Admissible}
 Suppose that there are infinitely many Siegel zeros. Then there are arbitrarily large $y$ for which there are admissible sets $A(y)$ of length $y$ with 
 \[
 A(y)\sim  \frac {2y}{\log y}.
 \]
  \end{cor}

 We can be more precise about sets that have many integers left unsieved:

  \begin{prop} \label{prop;Motohashi}
 Suppose that there is an infinite sequence of exceptional zeros $\beta$ corresponding to real primitive characters of conductor $q$, and let $z=y^u$ with $1\leq u\leq 3$.   Then there exists values of $X$ such that:
 \begin{itemize}
 \item  If  $1-\beta\leq \frac {\delta^2} {\log q}$ for some fixed $\delta>0$    then 
 \[
S(X,y,z) \geq    \frac {2y}{\log y}  - (2\delta\, C(u)+o(1))    \frac {y}{\log y}  
\]
where $C(u)=\sqrt{2(1-\log^+ (u-1) )}$;
\item  If   $1-\beta\leq \frac 1{(\log q)^\kappa}$ for some fixed $\kappa>1$ then 
\[
S(X,y,z)\geq  \frac {2y}{\log y}  - C_\kappa(u) (\log y)^{\tfrac 2{\kappa+1}} \frac {y}{(\log y)^2}     ;
\]
for some constant $C_\kappa(u) >0$;
\item If  $1-\beta \leq  \exp(-(\log q)^{1/\tau})$ for some fixed $\tau\geq 1$ then
\[
S(X,y,z)\geq  \frac {2y}{\log y}  -  c_\tau(\log \log y)^{\tau} \frac {y}{(\log y)^2}   ;
\]
for some constant $c_\tau >0$;
\item If $1-\beta\leq 1/q^{\epsilon}$ and $\epsilon\to 0$ slowly with $q$ then 
\[
S(X,y,z)\geq  \frac {2y}{\log y}  -  (2/\epsilon+o(1)) \frac {y\log \log y}{(\log y)^2}   .  
\]
 \end{itemize}
 \end{prop}

 One consequence of the first part is that if one can show that for all integers $x$ and $y$ sufficiently large we have
 \[
  S(x,y,y^{1/2}) \leq (2-\eta) \frac {y}{\log y}
 \]
 then  any real zeros $\beta$ of $L(s,\chi)$ for   a primitive quadratic character $\chi \pmod q$ satisfy 
 \[
 \beta\leq 1-\frac {\eta^2+o(1)} {8\log q}; 
 \]
 that is, there are no Siegel zeros.
 (Again, this is closely related to the work of  Motohashi \cite{Mo}.)

 In \cite{Se2}   Selberg noted that he could find examples for $u\leq 3$ with 
 \[
 S(\mathcal A,z)  \geq  \frac {2y}{\log y}\bigg( 1-  \frac {c(\log\log y)^2}{\log y} \bigg) 
 \]
 (which we obtain from Proposition \ref{prop;Motohashi} if $1-\beta = \exp(-(\log q)^{1/2+o(1)})$), and he states that he thought he could reduce the  $(\log\log y)^2$ to $\log\log y$ in the secondary term. Thanks to Siegel's Theorem this is (just) beyond the realms of possibility with our construction (see the last part of Proposition \ref{prop;Motohashi}).
 
 It  is feasible that the limits
 \[
 \lim_{y\to \infty} \max_x S(x,y,y^{1/u})\bigg/\frac {y}{\log y} \text{ and }  \lim_{y\to \infty} \min_x S(x,y,y^{1/u})\bigg/\frac {y}{\log y}
 \]
might not exist; indeed if the extremal examples all come from Siegel zeros (as in this paper), and if Siegel zeros are very spaced out (as we might expect if they do exist), then these limits will not exist. Therefore, one  needs to work with $\limsup$ and $\liminf$, respectively, in this kind of formulation of our results
(Selberg \cite{Se2} made the analogous point about $S(\mathcal A,z)$.)

 Tao wrote in his blog:\footnote{https://terrytao.wordpress.com/2007/06/05/open-question-the-parity-problem-in-sieve-theory/}
 ``\textit{The parity problem can also be sometimes overcome when there is an exceptional Siegel zero} ... [this] \textit{suggests that to break the parity barrier, we may assume without loss of generality that there are no Siegel zeroes}'' (see also section 1.10.2 of \cite{Tao7}). The results of this article suggest that this claim needs to be treated with caution, since its truth depends on the context.


 \section{Exceptional zeros}

 Landau proved that  there exists a constant $c>0$ such that if $Q$ is sufficiently large then there can be no
more than one modulus $q\leq Q$, one primitive real character $\chi \pmod {q}$
and one real number $\beta$ for which $L(\beta,\chi)=0$, with
$$
\beta \geq  1 - \frac c{\log Q} \ .
$$
These are the so-called \emph{exceptional zeros} (or \emph{Siegel zeros}); we do not believe that they exist (as they would contradict the Generalized Riemann Hypothesis), but they are the most egregious putative zeros that we cannot discount.

If we  assume that exceptional zeros exist then there are some surprising but simple organizing principles:
\begin{itemize}
\item We can assume that if they exist then there are infinitely many, else we simply change the value of $c$ to  $c=\frac 12\min_q (1-\beta)\log q$ and then there are no exceptional zeros.
\item  We may also assume that one can take $c$ arbitrarily small for if one does not have exceptional zeros for a small enough $c$ then we are done.
\end{itemize}

So henceforth we will assume  that \emph{there are infinitely many Siegel zeros}; that is, for any $\kappa >0$  arbitrarily small, there is a sequence $(q_j,\chi_j,\beta_j)_{j\geq 1}$ such that 
\begin{equation} \label{eq: Skonstant}
\beta_j \geq  1 - \frac \kappa{\log q_j} \text{ for all } j\geq 1.
\end{equation}
We now plug this zero into the explicit formula for primes in arithmetic progressions. To do this we slightly modify section 18.4 of \cite{IK}:

\begin{lemma} \label{lem: pntIK}
There exists a (large) constant $A$ such that if 
there is a Siegel zero $\beta$ of a real quadratic character mod $q$ and
\[
q^A \leq x\leq e^{1/(1-\beta)}
\]
with $(a,q)=1$ then 
\begin{equation} \label{eq: PNTAP.IK}
\phi(q) \psi(x;q,a) = (1-\chi(a))x +   (1-\beta)  (\chi(a)+  O( (1-\beta)\log x )) x(\log x-1)
\end{equation}
\end{lemma}

We deduce, by partial summation and the prime number theorem, that 
\begin{equation} \label{eq: PNTAP.IK2}
\phi(q) \pi(x;q,a) = (1-\chi(a))\pi(x) +   (1-\beta)  (\chi(a)+  O( (1-\beta)\log x )) x 
\end{equation}
for $q^{A+c}\leq x \leq e^{1/(1-\beta)}$.

\begin{proof}  We let $A=2(c_2+1/c_3)$ with $c_2, c_3$ as in \cite{IK} and then select $T$ so that 
$x=(qT)^{A/2}$, which ensures that $T\geq q$. Then (18.82) of \cite{IK} (with our value of $T$ replacing theirs) yields that
\[
\phi(q) \psi(x;q,a) = x-\chi(a) \frac{x^\beta}\beta + O( ((1-\beta)\log x)^2x+   x^{1-c_0} )
\]
for some constant $c_0>0$, since $(\frac{(qT)^{c_2}}x)^\eta =(qT)^{-\eta/c_3}=(1-\beta)\log qT\ll (1-\beta)\log x$
in the calculation in \cite{IK} with $\eta=c_3\frac{|\log((1-\beta)\log qT)|}{ \log qT}$ as in (18.11) of \cite{IK}.
Therefore \eqref{eq: PNTAP.IK} follows since Siegel's theorem implies that $1-\beta\gg_\epsilon q^{-\epsilon}\geq x^{-\epsilon}$, so our second error term is smaller than the first, and then we estimate the main term from its Taylor series.
\end{proof}

\begin{cor}  
 There exists a (large) constant $A$ such that  if
there is a Siegel zero $\beta$ of a real quadratic character $\chi$ mod $q$ satisfying $\beta \geq  1 - \frac \kappa{\log q} $
 then for any 
 \begin{equation} \label{eq: range}
x\geq q^A \text{ with } (1-\beta)\log x\leq \Delta =\Delta_\beta(x)
\end{equation}
and  $(a,q)=1$ we have
\begin{equation} \label{eq: PNTAPs}
\pi(x;q,a)= \frac{\pi(x)}{\phi(q)}  \cdot \begin{cases}
O(\Delta) & \text{ if }  \chi(a)=1;\\
2+O(\Delta) &\text{ if }  \chi(a)=-1.
\end{cases}
\end{equation}
In particular, for fixed $\Delta\in (0,1]$ and  $B>A$   we can take $\kappa=\Delta/B$ and \eqref{eq: PNTAPs} holds for all $x$ in the range $q^A\leq x\leq q^B$.
\end{cor}

\begin{proof} The first part follows immediately from \eqref{eq: PNTAP.IK2} and the prime number theorem.
For the last part note that $(1-\beta)\log x\leq  \kappa\frac {\log x} {\log q}  \leq \kappa B=\Delta$ by hypothesis.
\end{proof}

 Let $\pi_{k}(x,z)$   be the number of integers  $n\leq x$ where $n=p_1\cdots p_k$ and 
 the $p_i$ are all primes $>z$; and 
 $\pi_{k}(x,z;q,a)$ to be the number of these integers that are $\equiv a \pmod q$.
 Let $\pi_{k,q}(x,z)$ be those that are coprime to $q$.
 These definitions are of interest when $x>2z^k$ in which case one can deduce immediately from the prime number theorem that 
 if $x\leq z^{O(1)}$ then
 \begin{equation} \label{eq: EasyBd}
 \pi_{k,q}(x,z) \asymp_k \frac x{\log z}.
  \end{equation}
 
 \begin{theorem} \label{thm: pik} There exists a (large) constant $A$ such that 
  if there is a Siegel zero $\beta$ of a real quadratic character $\chi$ mod $q$ satisfying $\beta \geq  1 - \frac \kappa{\log q} $
 then for any $z$ in the range \eqref{eq: range}, and any fixed $u>1$, if $k$ is a positive integer $<u$ with $x=z^u$ and $(a,q)=1$ then  we have, uniformly,
 \[
\pi_{k}(x,z;q,a)= (1+(-1)^k\chi(a)+ O(\Delta_\beta(x)))\frac{\pi_{k,q}(x,z)}{\phi(q)}   .
\] 
 \end{theorem}
 
  \begin{proof} We saw above that this is true for $k=1$. Then we proceed by induction on $k>1$. We write $n=mp$ so that 
\[
\pi_{k}(x,z;q,a) = \frac 1k \sum_{\substack{ z<p\leq x/z^{k-1}\\ p\nmid q}} \pi_{k-1}(x/p,z;q,a/p) + \text{Error} ;
\]
where the error   comes from terms on the right-hand side  of the form $rp^2$, and
$r$ is counted by $\pi_{k-2}(x/p^2,z;q,a/p^2)\ll  x/{qp^2}$.  Therefore,  
\[
\text{Error} \ll  \frac 1k \sum_{p>z} \frac{x}{qp^2} \ll \frac x{kq z\log z} \ll \frac{\pi_{k,q}(x,z)}{z\phi(q)} \ll \Delta \frac{\pi_{k,q}(x,z)}{\phi(q)} 
\]
by Siegel's Theorem as $1/z\leq q^{-A}\ll 1-\beta \leq \Delta:=\Delta_\beta(x)$.
By the induction hypothesis, we then have 
\[
\pi_{k}(x,z;q,a) = \frac 1k \sum_{\substack{ z<p\leq x/z^{k-1}\\ p\nmid q}}(1-(-1)^{k}\chi(a/p)+ O(\Delta))\frac{\pi_{k-1,q}(x/p,z)}{\phi(q)} +O\bigg(\Delta   \frac{\pi_{k,q}(x,z)}{\phi(q)} \bigg) .
\]
Summing over $(a,q)=1$ we also have
\[
 \frac 1k \sum_{\substack{ z<p\leq x/z^{k-1}\\ p\nmid q}} \pi_{k-1,q}(x/p,z) = (1 + O(\Delta))\pi_{k,q}(x,z) .
\]
We can obtain this complete sum from our equation for $\pi_{k}(x,z;q,a)$ plus some extra terms as follows:
\begin{align*}
\pi_{k}(x,z;q,a) &= \frac {1+(-1)^{k}\chi(a)+ O(\Delta)} k \sum_{\substack{ z<p\leq x/z^{k-1}\\ p\nmid q}} \frac{\pi_{k-1,q}(x/p,z)}{\phi(q)} \\
&\qquad - 2\frac {(-1)^{k}\chi(a) }k  \sum_{\substack{ z<p\leq x/z^{k-1}\\ \chi(p)=1}} \frac{\pi_{k-1,q}(x/p,z)}{\phi(q)} +O\bigg(\Delta   \frac{\pi_{k,q}(x,z)}{\phi(q)} \bigg) .
\end{align*}
The first line of the right-hand side gives the main term in the result, and the last error term is acceptable.
By \eqref{eq: EasyBd} the second sum is 
\[
\ll_k  \sum_{\substack{ z<p\leq x/z^{k-1}\\ \chi(p)=1}}      \frac {x/p} {\phi(q) \log z} \ll  \frac {x} {\phi(q) \log z} \cdot \Delta \log \bigg( \frac{\log x/z^{k-1} }{\log z } \bigg)  \ll  \Delta  \frac {\pi_{k,q}(x,z)} {\phi(q)  }
\]
using partial summation on \eqref{eq: PNTAPs} (for each $a$ with $\chi(a)=1$)  to obtain the upper bound on the sum over primes, and then \eqref{eq: EasyBd} for the last inequality, and since $x=z^{O(1)}$.   This is acceptable in our error term, and the error term can be given uniformly since we iterate this process only finitely often because of the range for $k$.
  \end{proof}

 Now let $N(x,z)$   be the number of integers  $n\leq x$ all of whose prime factors are $>z$, 
 $N(x,z;q,a)$   be the number of these integers that are $\equiv a \pmod q$, and $N_q(x,z)$ those coprime with $q$.
 Let $N(x,z)^\pm$ count these integers with $\lambda(n)=\mp 1$, and then 
  $N_{q}(x,z)^\pm$  those coprime with $q$.  
  
  \begin{cor} \label{cor: 1stex}
  There exists a (large) constant $A$ such that for any constants $u>1$ which is not an integer and $\Delta>0$, assume there are
infinitely many Siegel zeros $\beta$ of   real quadratic characters $\chi$ mod $q$ satisfying $\beta \geq  1 - \frac \kappa{\log q} $
with $\kappa=\Delta/Au$.   For $z=x^{1/u}=q^A$  we have, uniformly,
  if $\chi(a)=-1$ then
 \[
N(x,z;q,a)=     ( F(u) + O(\Delta ))  \cdot \frac{1}{\phi(q)}  \prod_{p\leq z} \bigg( 1 -\frac{1}p \bigg) x
\] 
 and  if $\chi(a)=1$ then
 \[
N(x,z;q,a)=    ( f(u) + O(\Delta)    )  \cdot \frac{1}{\phi(q)}  \prod_{p\leq z} \bigg( 1 -\frac{1}p \bigg) x.
\] 
 \end{cor}
 
  \begin{proof}  Theorem \ref{thm: pik}  implies that 
  \[
N(x,z;q,a)=  \frac{1}{\phi(q)}   ( 2N_q(x,z)^\pm + O(\Delta   N_q(x,z) )
\] 
since $N(x,z;q,a)=\sum_{k\leq u} \pi_{k}(x,z;q,a)$ and $N_{q}(x,z)=\sum_{k\leq u} \pi_{k,q}(x,z)$.
  Now the integers $n$ counted in $N(x,z)$ belong to congruence classes coprime to $q$, and are already coprime to all of the primes $\leq z$ that do not divide $q$. Since $q<z$ we deduce that $N_q(x,z)^\pm=N(x,z)^\pm$, and $N_q(x,z) =N(x,z) $
  
  Now $N(x,z)^\pm=S(\mathcal A^\pm,z)$ and so, applying \eqref{eq: Selbergexample} we have
 \[
 N(x,z)^- \sim f(u) \cdot \prod_{p\leq z} \bigg( 1 -\frac{1}p \bigg)  \frac x2 \text{ and } N(x,z)^+ \sim F(u) \cdot \prod_{p\leq z} \bigg( 1 -\frac{1}p \bigg)  \frac x2 
 \]
 as $ \# \mathcal A^+, \# \mathcal A^-\sim \frac x2$ by the prime number theorem.
 \end{proof}

   \begin{proof} [Proof of Corollary \ref{cor: SiegelGeneralv}]
   We select  triples $q,qy$ (in place of $x$) and $z$ as in Corollary \ref{cor: 1stex}.
 Let $P_q(z)=\prod_{p\leq y,\ p\nmid q} p$ and select $r$ so that $rq\equiv 1 \pmod {P_q(z)}$, and then let $b\equiv ra\pmod {P_q(z)}$. Then
 $b+j\equiv r(a+jq) \pmod  {P_q(z)}$ so that $(b+j,P_q(z))=(a+jq,P_q(z))$. Let $S:=\{ j\in [0,y-1]: (b+j,P_q(z))=1\} $ so that
 \[
 \# S = \#\{ j\in [0,y-1]: (a+jq,P_q(z))=1\} = N(qy,z;q,a).
 \]
 Therefore if $\chi(a)=-1$ then 
 \[
 \# S   = ( F(u) + O(\Delta ))  \cdot \frac{q}{\phi(q)}  \prod_{p\leq z} \bigg( 1 -\frac{1}p \bigg) y
 \]
 by Corollary \ref{cor: 1stex}.
 
If $q=p_1^{e_1}\cdots p_k^{e_k}$ with select $a_1\pmod {p_1}$ to minimize $\{ s\in S: s\equiv a_1\pmod {p_1}\}$ and let
$S_1:=\{ s\in S: s\not\equiv a_1\pmod {p_1}\}$ so that $\# S_1\geq (1-\frac 1{p_1}) \# S$. We then do the same for $p_2,\dots$  and select an integer $B$ for which $B\equiv b \pmod  {P_q(z)}$ and $B\equiv -a_j \pmod {p_j}$ for all $j$. Therefore
\[
\#\{ j\in [0,y-1]: (B+j,P(z))=1\}\geq \frac{\phi(q)}q  \# S   \geq ( F(u) + O(\Delta ))    \prod_{p\leq z} \bigg( 1 -\frac{1}p \bigg) y.
\]
Now $qy=z^u$ and $y=z^v$ so that as $z=q^A$ we have $u=v+\frac 1A$.  
Therefore 
\[
F(u) + O(\Delta )=F\bigg(v+\frac 1A\bigg) + O(\Delta )=F(v)+o_{A\to \infty}(1),
\]
since $F(.)$ is continuous. Therefore, letting $A\to \infty$ we have
\[
 \# \{ n\in (x,x+y]: (n,P(y^{1/v}))=1 \} \gtrsim F(u)  \prod_{p\leq z} \bigg( 1 -\frac{1}p \bigg) y,
 \]
 A lower bound of the same size follows from   Jurkat and Richert's result given in \eqref{eq: JR theorem}, and therefore the asymptotic result follows.

If $\chi(a)=1$ then we proceed analogously but instead we select our arithmetic progressions $a_j \pmod {p_j}$ to maximize the number in this arithmetic progression in the already-sifted set.
 \end{proof}

\begin{proof} [Proof of Corollary \ref{cor: Admissible}] Fix $\epsilon>0$.
Take $v=1/(1-\epsilon)$ in Corollary \ref{cor: SiegelGeneralv}] so that there exists an  integer $x$ for which
\[
 S(x,y,y^{1-\epsilon}) \sim 2e^\gamma (1-\epsilon) \prod_{p\leq y^{1-\epsilon}} \bigg( 1-\frac 1p \bigg) y \sim \frac {2y}{\log y}.
\]
Let $B$ be the set of positive integers $n\leq y$ for which $x+n$ has no prime factor $\leq y^{1-\epsilon}$ so that 
$\# B=S(x,y,y^{1-\epsilon})$, and $B$ contains no integers $\equiv -x \pmod p$ for every prime $p\leq y^{1-\epsilon}$.
If the primes in $ (y^{1-\epsilon},y]$ are $p_1<p_2<\dots<p_k$, we let $B_1=B$ and then for $j=1,\dots,k$ we select the arithmetic progression $a_j \pmod {p_j}$ containing the least number of elements of $B_j$, and let $B_{j+1}$ be $B_j$ less that arithmetic progression. Therefore $\# B_{j+1}\geq (1-\frac 1{p_j})\#B_j$ for each $j$, and so $A:=B_{k+1}$ is an admissible set with
\[
\frac {2y}{\log y} \sim \# B\geq \# A \geq \prod_{ y^{1-\epsilon}<p\leq y} \bigg( 1-\frac 1p \bigg) \# B  \gtrsim (1-\epsilon)\frac {2y}{\log y};
\]
that is, $\# A=(2+O(\epsilon))\frac {y}{\log y} $.
The result follows letting $\epsilon\to 0^+$.
\end{proof}

 \section*{At the sifting limit, redux}

 The link between exceptional zeros and the sifting limit was discussed in section 9 of \cite{BFT}. We now develop these ideas further when $y=z^{2+o(1)}$.  Here we will take  $y=q^{A-1}$ with $A\to \infty$ as slowly as we like.

 \begin{proof} [Proof of Proposition \ref{prop; Iwaniec}]
 We will assume that $\Delta=\Delta_\beta(x)\leq 1$.
 By \eqref{eq: PNTAP.IK2}, for $x\geq q^A$ we have
 \[
 \pi(x;q,a) = (1-\chi(a)) \frac{ \pi(x)}{\phi(q)} +  (\chi(a)+O(\Delta) ) (1-\beta) \frac{x}{\phi(q)}  .
 \]
 Proceeding as in the proof of Corollary \ref{cor: SiegelGeneralv}, we have, for $b\equiv a/q \pmod {P_q(z)}$,
 \[
\#\{ j\in [0,y-1]: (b+j,P_q(z))=1\} = N(x,z;q,a)  .
 \]
 If $x^{1/2}<z\leq x$ and $\chi(a)=1$ then the above implies that
 \[
 N(x,z;q,a)=  \pi_{1}(x,z;q,a) =\pi(x;q,a) -\pi(z;q,a) = (1+O(\Delta+\lambda) ) (1-\beta) \frac{x}{\phi(q)}   
 \]
 where $\lambda=\lambda(z,x):= \frac 1{\log x}+ \frac zx$.
 Letting $x=qy$ and selecting residue classes for each prime dividing $q$ as in the proof of Corollary \ref{cor: SiegelGeneralv}, we deduce that there exists an integer $B$ for which
 \[
 \#\{ j\in [0,y-1]: (B+j,P(z))=1\} \leq \frac{\phi(q)} q N(x,z;q,a)  \leq (1-\beta) y (1+O(\Delta+\lambda) ).
 \]
 
 In the range $x^{1/2}\geq z > x^{1/3}$ we have $N(x,z;q,a)=  \pi_{1}(x,z;q,a)+ \pi_{2}(x,z;q,a)$, and  
 \begin{align*}
 \pi_{2}(x,z;q,a)&=   \sum_{\substack{z<p<\sqrt{x} \\ p\nmid q}}  \pi_{1}(x/p;q,a/p)- \pi_{1}(p;q,a/p) \\
 & =  \sum_{\substack{z<p<\sqrt{x} \\ p\nmid q}} \left(  (1-\chi(a/p)) \frac{ \pi(x/p)-\pi(p)}{\phi(q)} +  (\chi(a/p)+O(\Delta)) (1-\beta) \frac{x/p-p}{\phi(q)} \right).
\end{align*}
Summing over all $(a,q)=1$ we obtain
\begin{align*}
\pi_{2,q}(x,z)& =  \sum_{\substack{z<p<\sqrt{x} \\ p\nmid q}}  \pi(x/p) -\pi(p)  +  O\bigg( \Delta  (1-\beta) \frac{x}{p }  \bigg)\\
&= \sum_{\substack{z<p<\sqrt{x} \\ p\nmid q}} ( \pi(x/p) -\pi(p) ) +  O(\Delta(1-\beta) x)
\end{align*}
Therefore $ \phi(q) \pi_{2}(x,z;q,a)   -  (1+\chi(a))\pi_{2,q}(x,z)$ equals $\chi(a)$ times
\[
=-2  \sum_{\substack{z<p<\sqrt{x} \\ \chi(p)=1}}    ( \pi(x/p)-\pi(p)) +  (1-\beta)  \sum_{ z<p<\sqrt{x}  }  \chi(p)  (x/p-p)  +  O(\Delta(1-\beta) x)
\]
Since $\#\{ p\leq x:\chi(p)=1\} = \frac 12(1+O(\Delta)) (1-\beta)x$ we deduce by partial summation that 
 \begin{align*}
2 \sum_{\substack{z<p<\sqrt{x} \\ \chi(p)=1}}    ( \pi(x/p)-\pi(p))  &= (1-\beta) \int_z^{\sqrt{x}}  \frac x{t\log(x/t)} dt+  O(\Delta(1-\beta) x)\\
&=  (1-\beta)x \bigg(\log (u-1)-\log \frac u2  +  O(\Delta)\bigg)    ,
\end{align*}
for $x=z^u$ while
 \begin{align*}
\sum_{ z<p<\sqrt{x}  }  \chi(p)  (x/p-p)  = -x \log\bigg( \frac{\log \sqrt{x}}{\log z}\bigg) +O((\Delta+\lambda)x)
=   -x  \bigg( \log \frac u2  +  O(\Delta+\lambda)\bigg)  .
\end{align*}
So if $x=z^u$ with $u\geq 2$ then 
\[
\phi(q) \pi_{2}(x,z;q,a)   =  (1+\chi(a))\pi_{2,q}(x,z)  -\chi(a)(1-\beta) x (\log (u-1)+O(\Delta+\lambda)).
\]
Therefore if $2\leq u<3$ then
 \begin{align*}
  N(x,z;q,a)&=  \pi_{1}(x,z;q,a)+ \pi_{2}(x,z;q,a)\\
&=(1-\chi(a)) \frac{ \pi(x)}{\phi(q)}   + (1+\chi(a))  \frac{\pi_{2}(x,z)}{\phi(q)}  +\chi(a)(1-\beta) \frac{x}{\phi(q)} (  1-\log (u-1)+O(\Delta+\lambda)) .
\end{align*}
In this range for $x$ we have $\lambda\asymp \frac 1{\log x}$.
 We use the prime number theorem to obtain
\[
 \pi_{2,q}(x,z)=\pi_{2}(x,z)\sim
\begin{cases} \frac x{\log x} \log(u-1) &\text{ if } x/z^2\to \infty;\\
\frac {2x}{(\log x)^2} (\log c -(1-\frac 1c)) &\text{ if } x=cz^2,\ c>1.
\end{cases}
\]
The latter estimate yields that if $x>z^2$ then 
\[
\pi_{2}(x,z) \lesssim \frac {2x\log x/z^2}{(\log x)^2} \leq \frac {2qy\log qy/z^2}{(\log y)^2}
\]
when $x=qy$. If 
$\chi(a)=1$ then, as above, there exists an integer $B$ for which
 \begin{align*}
S(B,y,z) &\leq \frac{\phi(q)} q N(x,z;q,a)  =   2\frac{\pi_{2}(qy,z)}{q}  + (1-\beta) y (  1-\log (u-1)+O(\Delta+\lambda)) \\
&\lesssim  \frac {4y\log qy/z^2}{(\log y)^2} + (1-\beta) y (  1+O(\Delta))
\end{align*}
since $y\gg z^2$.

The claimed result now follows: For any $\epsilon>0$ we let $y=q^{1/\epsilon-1}$ and $\kappa=\epsilon^2$ so that $\Delta=(1-\beta)\log qy\leq  \frac{\kappa}{\log q}\cdot \epsilon^{-1}  \log q =\epsilon$, and is therefore arbitrarily small.
 \end{proof}
 
 \begin{proof} [More than the proof of Proposition \ref{prop;Motohashi}]
Taking $\chi(a)=-1$ in the previous proof we obtain that
 if $x^{1/3}<z\ll \frac x{\log x}$ then 
 \[
 \phi(q) N(x,z;q,a) =  (2+O(\lambda)) \frac x{\log x}  -   (1-\beta) x (  1-\log^+ (u-1)+O(\Delta+\lambda)) .
\]
Proceeding as before (but now removing the arithmetic progressions with the least number of unsieved elements),  
 this implies that there exists an integer $X$ for which
 \begin{align*}
 S(X,y,z) &\geq  (2+O(\lambda)) \frac y{\log qy}  -  \frac 12 (1-\beta) y (  C(u)^2+O(\Delta+\lambda)) \\
 &\geq   \frac {2y}{\log y} \bigg( 1-\frac{\log q+O(1)}{\log y} +O\bigg( \bigg( \frac{\log q}{\log y} \bigg)^2 \bigg) \bigg)  -   \frac 12 (1-\beta) y (  C(u)^2+O(\Delta+\lambda)) \\
  &\geq   \frac {2y}{\log y} \bigg( 1-\frac 1A -   \frac { C(u)^2}{4B} +O\bigg(\frac 1{A\log q}+  \frac 1{A^2}+  \frac 1{B^2} \bigg)   \bigg) 
  \end{align*}  
  writing $y=q^A$ and $1-\beta=\frac 1{B\log y}$ so that $\Delta\ll \frac 1B$. We select 
   \[ A=\frac 2{C(u)} \sqrt{\frac 1{(1-\beta)\log q}} \text{ and }  B=\frac {C(u)}2 \sqrt{\frac 1{(1-\beta)\log q}}\]
    so that
 \begin{align*}
S(X,y,z) &\geq  \frac {2y}{\log y} \bigg( 1-C(u)   \sqrt{  (1-\beta)\log q}  +O\bigg( \sqrt{  \frac{1-\beta}{\log q}}+  (1-\beta)\log q \bigg)   \bigg) \\
& =  \frac {2y}{\log y} -  (1+o(1))      \frac {4y\log q}{(\log y)^2}
  \end{align*}  
since $  C(u)   \sqrt{  (1-\beta)\log q} = \frac {2 \log q} {\log y}$.
 The results claimed in the proposition therefore follow by inserting the given values for $1-\beta$ and determining $\log q$ in terms of $y$.
  \end{proof}

 \section{Largest gaps between primes, when there are Siegel zeros}
 
 \begin{proof} [Proof of Corollary \ref{cor: PrimeGaps}]
 In the proof of Proposition  \ref{prop; Iwaniec} we have seen that there exists and integer $X$ for which
 $N:=S(X,y,z) \lesssim (1-\beta) y$ where $z=(qy)^{1/2}$.
 Let $a_1,\dots,a_N$ be the integers in $\{ X+a\in (X,X+y]: \   (X+a,P(z))=1 \}$, and
 $p_1<p_2<\dots<p_N$ be primes taken from the interval $(z,Z]$ where $Z:=(1+\epsilon)(1-\beta) y\log y$.
 This is possible since   $z=(qy)^{1/2}=o((1-\beta) y)$ (as $1-\beta\gg q^{-o(1)}$ by Siegel's theorem, and $y=q^A$ with $A$ large) so, by the prime number theorem there are more than $N$ primes in the interval.
 
 We now select an integer $x$ such that $x\equiv X \pmod {P(z)}$ and $x\equiv -a_j \pmod {p_j}$ for $1\leq j\leq N$. We see that 
 $S(x,y,Z)=0$. Therefore if $p_n$ is the largest prime $\leq x$ and $p_{n+1}$ is the next smallest prime then $p_{n+1}-p_n>y$ while we can select $x\in (P(Z),2P(Z)]$. Therefore we see that $Z\sim \log x$ by the prime number theorem, and 
 \[ y\sim \frac{ Z/(1-\beta)} {\log Z}\] letting $\epsilon\to 0^+$. We deduce that 
 $A\log q\sim \log y\sim \log Z\sim \log\log x$.
 Therefore if $1-\beta=\frac 1{(\log q)^B}$ then $\frac 1{1-\beta} =(\log q)^B\sim A^{-B} (\log\log x)^B$, and so
 \[
 y\sim A^{-B} \log x (\log\log x)^{B-1}.
 \]
 The result follows taking $A$ fixed but large.
 \end{proof}

 Cram\'er conjectured that the largest gap between primes $\leq x$ should be $\sim (\log x)^2$; however 
 Cram\'er made this conjecture based on a model for the distribution of primes that does not take into account divisibility by
 small primes. A modified model is discussed in \cite{GL} (as well as \cite{BFT}) which does take into account the small primes.
 Proposition 1 of \cite{GL} suggests that we take $z=\epsilon \log x$ and $y=y(x)$ a little larger than $(\log x)^2$ so that 
 \[
 \min_X S(X,y,z) \sim \prod_{p\leq z} \bigg( 1 -\frac 1p\bigg)  (\log x)^2
 \]
 then the largest gap between primes in $[x,2x]$ should be $\sim y(x)$.  By making certain guesses about the sieve, the authors of  \cite{GL} then predict that $y(x)\sim 2e^{-\gamma}(\log x)^2$ (or perhaps with a slightly larger constant than ``$2e^{-\gamma}$'', depending on a certain sieve constant). However we see have seen here that the existence of Siegel zeros plays havoc on our guesses about sieving intervals. If we use Mertens' Theorem and substitute in the bound from Proposition
 \ref{prop; Iwaniec} then we obtain, writing $y=w(\log x)^2$ where $w=w(x)\geq 1$,
 \[
     w ( \log (qw/\epsilon^2 )+ (1-\beta_q)(\log\log x)^2   ) \gtrsim e^{-\gamma}  \log\log x
 \]
 where  $y=q^{A-1}$ with $A\to \infty$ as slowly as we like. 
 Now $\log q\ll \frac 1A \log\log x$ and so we see from here that $w(x)\to \infty$, letting $A\to \infty$. The question is how fast?
 
 The proof of Proposition
 \ref{prop; Iwaniec} give  $\log y \ll \frac 1{1-\beta}$   and so 
 $\log\log  x\ll \log y\ll q^{o(1)}$ by Siegel's Theorem
 We will be unable to prove $w$ to be any larger than $\log\log x\ll q^{o(1)}$ so the above inequality can be taken to be 
 \[
     w ( 4\log q + (1-\beta_q)(\log y)^2   ) \gtrsim 2e^{-\gamma}  \log y.
 \]
 We optimize by taking $\log y=2 \sqrt{ \frac{\log q}{1-\beta} }$ and so this becomes
 \[
      w   \gtrsim       \frac {e^{-\gamma} \log y}{4 \log q}
 \]
 What does this mean in terms of $x$?
 
 If we can only say that there are infinitely many Siegel zeros than this heuristic implies that 
 \[
\limsup_{x\to \infty}  \frac{ \max_{p_n\leq x} p_{n+1}-p_n }{(\log x)^2} \to \infty.
 \]
 Suppose instead that we have infinitely many $q$ for which $1-\beta\ll 1/(\log q)^{2c-1}$ for some $c>1$. Then our heuristic implies that
there are infinitely $p_n$ for which
\[
p_{n+1}-p_n  \gg (\log p_n)^2 (\log\log p_n)^{1-\tfrac 1c}.
\]
If $1-\beta<1/\exp(\log q)^{1/c})$ for some $c>1$ then our heuristic implies that
there are infinitely $p_n$ for which
\[
p_{n+1}-p_n  \gg (\log p_n)^2 \frac{ \log\log p_n}{(\log\log\log p_n)^c}.
\]
Much the same predictions can be deduced from the heuristic in \cite{BFT}, as pointed out to me by Ford.

 \bibliographystyle{plain}

\begin{thebibliography}{99}

\bibitem{BFT} William Banks, Kevin Ford and Terence Tao,  
\emph{Large prime gaps and probabilistic models},
(preprint).

\bibitem{Bo}  Enrico Bombieri,  
\emph{The asymptotic sieve},
 Rend. Accad. Naz. XL \textbf{1} (1975/76), 243--269 (1977).

 \bibitem{FGM}  Kevin Ford, Ben Green, Sergei Konyagin, James Maynard,  and Terence Tao,
 \emph{Long gaps between primes},
J. Amer. Math. Soc.  \textbf{31} (2018),   65--105. 

\bibitem{For}  Kevin Ford, 
\emph{Large prime gaps and progressions with few primes},
 Rivista di Matematica della Universita di Parma (to appear).
 
\bibitem{GL} Andrew Granville and Allysa Lumley, 
\emph{Primes in short intervals: Heuristics and calculations}
(preprint).

\bibitem{Iw1}  Henryk Iwaniec, 
\emph{On the error term in the linear sieve}, 
Acta Arith. \textbf{19} (1971), 1--30

 \bibitem{Iw2}  Henryk Iwaniec, 
\emph{On the problem of Jacobsthal}, 
Demonstratio Math. \textbf{11} (1978), 225--231.

\bibitem{IK}  Henryk Iwaniec and Emmanuel Kowalski, 
\emph{Analytic number theory}, 
AMS Colloquium Publications \textbf{53} (2004).


  
 \bibitem{JR}  W.B. Jurkat and H.-E. Richert,  
\emph{An improvement of Selberg's sieve method. I},
Acta Arith. \textbf{11} (1965), 217--240.

\bibitem{Mo}  Yoichi Motohashi, 
\emph{A note on Siegel's zeros},
Proc. Japan Acad. Ser. A Math. Sci.  \textbf{55} (1979), 190--191.

 \bibitem{Se1}  A. Selberg,  
\emph{Sieve methods},
ch. 36 of Collected Works, Vol I, Springer-Verlag, New York 1989.
\textsl{published originally as:} Proc. Sympos. Pure Math \textbf{20} (1971), 311--351.

 \bibitem{Se2}  A. Selberg,  
\emph{Remarks on sieves}, 
ch. 37 of Collected Works, Vol I, Springer-Verlag, New York 1989.
\textsl{published originally as:} Proc. 1972 Colorado Number theory conference, 205--216.

\bibitem{Se3}  A. Selberg,  
\emph{Lectures on sieves}, 
ch. 45 of Collected Works, Vol II, Springer-Verlag, New York 1989. 

\bibitem{Tao7}  Terry Tao,  
\emph{Structure and Randomness: pages from year one of a mathematical blog}, 
American Math. Soc., Rhode Island, 2008. 


 \bibitem{Si}  H. Siebert,
\emph{Sieve methods and Siegel's zeros}, 
Studies in pure mathematics, 659--668, Birkh\"auser, Basel, 1983.

\end{thebibliography}

\end{document}